\documentclass[a4paper,oneside, reqno]{amsart}
\usepackage{mathrsfs}
\usepackage{amsfonts}
\usepackage{amssymb}
\usepackage{amsxtra}
\usepackage{dsfont}
\usepackage{amsthm}

\usepackage[english]{babel}

\usepackage{color}

\usepackage{enumitem}
\usepackage{hyperref}
\usepackage{stmaryrd} 

\newtheorem{theorem}[equation]{Theorem}

\newtheorem{lemma}[equation]{Lemma}
\theoremstyle{definition}

\numberwithin{equation}{section}

\theoremstyle{remark}

\newcommand{\RR}{\mathbb{R}}
\newcommand{\ZZ}{\mathbb{Z}}
\newcommand{\TT}{\mathbb{T}}

\newcommand{\QQ}{\mathbb{Q}}

\newcommand{\calF}{\mathcal{F}}

\newcommand{\ind}[1]{{\mathds{1}_{{#1}}}}

\def\R{{\mathbb{R}}}
\def\T{{\mathbb{T}}}
\def\C{{\mathbb{C}}}
\def\N{{\mathbb{N}}}
\def\Z{{\mathbb{Z}}}

\def\G{{\mathbb{G}}}

\DeclareMathOperator{\supp}{supp}

\title[Sharp bounds in Ionescu--Wainger multiplier theorem for canonical fractions]
{Sharp bounds in Ionescu--Wainger multiplier theorem for canonical fractions}

\author{Wojciech S{\l}omian}
\address{Wojciech S{\l}omian (\textnormal{wojciech.slomian@uga.edu})
\newline Department of Mathematics, University of Georgia, Athens, GA 30602, USA \& Instytut \linebreak Matematyczny, Uniwersytet Wroc{\l}awski, Plac Grunwaldzki 2/4, 50-384 Wro\-c{\l}aw, Poland}

\subjclass[2020]{42A45, 11P55}
\keywords{Ionescu--Wainger multiplier theorem, discrete analogues, circle method}
\thanks{The author was supported by AMS--Simons Travel Grant No. 00008780.}

\begin{document}
\begin{abstract}
In this short note, we show that the loss of order
$N^{c_p/\log\log N}$ in the version of the Ionescu--Wainger
multiplier theorem for canonical fractions, recently established by Kosz, Mirek,
S{\l}omian, and Zhang in \cite{KMSZ}, is optimal.
\end{abstract}
\maketitle
\section{Introduction}

Since the work of Bourgain \cite{B3} on ergodic averages along squares
and of Stein and Wainger \cite{SW1} on discrete Radon transforms,
the study of discrete analogues has grown into a well-established
and rapidly developing branch of harmonic analysis. Its deep
connections with number theory make the subject both intriguing
and challenging. The absence of arbitrarily small scales on $\Z$,
together with the limited applicability of tools such as
differentiation and changes of variables, makes it difficult
to adapt methods from the corresponding continuous setting on $\R$.

One of the most important tools in the study of discrete analogues is the Ionescu--Wainger multiplier
theorem, which, roughly speaking, provides an $\ell^p$ transference
principle for major-arc multipliers of the form
\begin{equation*}
    \Delta_{\varepsilon_N,\mathcal{A}}[\mathfrak{m}](\xi)
    := \sum_{a/q\in\mathcal{A}}
    \mathfrak{m}\!\left(\varepsilon_N^{-1}\left(\xi-a/q\right)\right),
    \qquad \xi\in\TT,
\end{equation*}
where $\mathfrak{m}$ is supported in $[-1/2,1/2]$,
$\varepsilon_N>0$ is sufficiently small, and
$\mathcal{A}\subset\QQ\cap[-1/2,1/2]$.

The Ionescu--Wainger multiplier theorem was first established
by Ionescu and Wainger \cite[Theorem 1.5]{IW} in their groundbreaking
work on discrete singular Radon operators. In their construction,
the major arcs were centered at what are now known as
Ionescu--Wainger fractions.

Since then, the theorem has become a widely used tool in discrete
harmonic analysis. It has also become an object of study in its own
right; see, for instance, the work of Tao \cite{TaoIW} and of Mirek,
Stein, and Zorin-Kranich \cite{MSZ3}. A more recent breakthrough came
in the work of Kosz, Mirek, Peluse, Wan, and Wright \cite{KMPWW}
on multilinear polynomial ergodic averages. These authors established
an Ionescu--Wainger multiplier theorem for the sets of
\textit{canonical fractions}, defined for each integer $N\geq 1$ by
\begin{equation*}
    \mathcal{R}_{\leq N}:=\left\{
    \frac{a}{q}\in\mathbb{Q}/\mathbb{Z}: a\in\mathbb{Z},\ q\in\{1,\ldots,N\},\ (a,q)=1\right\}.
\end{equation*}
In \cite[Theorem 3.3]{KMPWW}, the authors established the following
estimate:
\begin{equation*}
   \bigl\|
       \mathcal{F}_{\Z}^{-1}
       \bigl(
           \Delta_{\varepsilon_N,\mathcal{R}_{\leq N}}[\mathfrak{m}]
           \calF_{\Z}f
       \bigr)
   \bigr\|_{\ell^{p}(\Z)}
   \lesssim_p
   N^{C_p\frac{\log\log\log N}{\log\log N}}
   \|f\|_{\ell^p(\ZZ)}.
\end{equation*}
More recently, the author, together with Kosz, Mirek, and Zhang
\cite[Theorem 1.9]{KMSZ}, improved the exponent in this bound
to $1/\log\log N$.

\begin{theorem}\label{thm:IW:upper}
Fix $p\in(1,\infty)$. There exists a constant
$C_p>0$ such that the following holds
for every $N\in\N$. Assume that
\begin{align}\label{eq:support}
    0<\varepsilon\leq (2pN^p)^{-1}
\end{align}
and let $\mathfrak m\colon\RR\to\C$ be a measurable function
supported in $[-1/2,1/2]$. Moreover, let us assume that
\begin{equation*}
    \|\mathcal{F}_{\R}^{-1}
    (\mathfrak m\mathcal{F}_{\R}f)\|_{L^p(\RR)}
    \lesssim_p\|f\|_{L^p(\RR)},
\end{equation*}
for all $f\in L^2(\RR)\cap L^p(\RR)$. Then
\begin{equation}\label{IW UPPER}
    \bigl\|\mathcal{F}_{\Z}^{-1}
        \bigl(\Delta_{\varepsilon,\mathcal{R}_{\leq N}}[\mathfrak{m}]\calF_{\Z}f\bigr)
    \bigr\|_{\ell^p(\Z)}\lesssim_p N^{C_p/\log\log N}\|f\|_{\ell^p(\Z)}
\end{equation}
for all $f\in\ell^2(\Z)\cap\ell^p(\Z)$.
\end{theorem}
For further details on the Ionescu--Wainger theory and its history,
we refer the reader to \cite{KMSZ} and the references therein.
A natural question, motivated by Tao's result \cite{TaoIW} for Ionescu--Wainger
fractions, was whether the constant in the Ionescu--Wainger theorem
for canonical fractions could be chosen independently of $N$.
In the same paper, we obtained a surprising negative answer
to this question; see \cite[Theorem 1.15]{KMSZ}.
More precisely, we proved the following result.
\begin{theorem} \label{thm:IW:lower} Let $p \in (1,\infty) \setminus \{2\}$. Then there exists $c_p \in (0,1)$ such that the following is true. Fix a nonzero Schwartz function $\mathfrak m \colon \R \to \C$ supported on $[-1/2,1/2]$ and let $(\vartheta_N)_{N \in\N}$ be a~given sequence of numbers $\vartheta_N \in (0,1)$. Then, for infinitely many $N \in \N$, we have
\begin{equation*}
\bigl\| \mathcal{F}_{\Z}^{-1}(\Delta_{\varepsilon_N,\mathcal{R}_{\leq N}}[\mathfrak{m}]\calF_{\Z}(\cdot))  \bigr\|_{\ell^{p}(\Z) \to \ell^{p}(\Z)}
\gtrsim_{p}
e^{(\log N)^{c_p}} = N^{1 / (\log N)^{1-c_p}}
\end{equation*}
with some $0 < \varepsilon_N \leq \vartheta_N$.
\end{theorem}
We observe that, for sufficiently large $N$,
\begin{equation*}
    N^{1/\log\log N}> N^{1/(\log N)^{1-c_p}}.
\end{equation*}
This raises a natural question: what is the optimal dependence
on $N$ of the constant in the Ionescu--Wainger multiplier theorem?
In this note, we answer this question by showing that the dependence
on $N$ in Theorem~\ref{thm:IW:upper} is sharp.
More precisely, we prove the following result.
\begin{theorem} \label{thm:IW:lower_new} Let $p \in (1,\infty) \setminus \{2\}$. Then there exists $c_p \in (0,1)$ such that the following is true. Fix a Schwartz function $\mathfrak m \colon \R \to \C$ supported on $[-1/2,1/2]$ with $\mathfrak{m}(0)=1$ and let $(\varepsilon_N)_{N \in \Z_+}$ be a~given sequence of numbers $0<\varepsilon_N \leq N^{-2}$. Then, for infinitely many $N \in \N$, we have
\begin{equation}\label{eq:main lower}
\bigl\| \mathcal{F}_{\Z}^{-1}(\Delta_{\varepsilon_N,\mathcal{R}_{\leq N}}[\mathfrak{m}]\calF_{\Z}(\cdot))  \bigr\|_{\ell^{p}(\Z) \to \ell^{p}(\Z)}
\gtrsim_p  N^{c_p/\log\log N}
\end{equation}
\end{theorem}
The rest of the paper is devoted to proving Theorem~\ref{thm:IW:lower_new}.
\section{Overview}
As in the proof of \cite[Theorem 1.15]{KMSZ}, we consider the case
where $N=q_1\cdots q_K$ is a product of distinct primes. The following
kernel appears in that proof (see \cite[p.~27]{KMSZ}),
\begin{equation}\label{eq:kernel operator}
    h_N(n):=\prod_{i=1}^K
    \bigl(q_i\mathds{1}_{q_i\mid n}-1\bigr),
    \qquad n\in\Z/N\Z.
\end{equation}
Since $N$ is a product of distinct primes, the Chinese remainder
theorem gives the identification
\begin{equation*}
    \Z/N\Z\simeq\prod_{i=1}^K(\Z/q_i\Z).
\end{equation*}
This allows us to factor the kernel $h_N$ into components of the form
\begin{equation*}
    q_i\mathds{1}_{q_i\mid n}-1=q_i\mathds{1}_{\{0\}}(n)-1,\qquad n\in\Z/q_i\Z.
\end{equation*}
Each such function is the convolution kernel of the operator $I-E_{q_i}$ on
$\Z/q_i\Z$ with the normalized counting measure, where $I$ denotes the identity operator and
$E_{q_i}$ is the averaging operator defined by
\begin{equation*}
    E_{q_i}f(x):=\frac{1}{q_i}\sum_{y\in\Z/q_i\Z}f(y),
    \qquad x\in\Z/q_i\Z.
\end{equation*}
This factorization of the kernel was also observed in \cite{KMSZ}.
There, the lower bound was obtained by testing the operator
$I-E_q$ on $f=\delta_0=\mathds{1}_{\{0\}}$, which gives
\begin{equation*}
    \big\|(I-E_{q_i})\delta_0\big\|_{\ell^p(\Z/q_i\Z)}^p
    =(1-q_i^{-1})^p+q_i^{-p}(q_i-1)
    \gtrsim_p (1-q_i^{-1})^p(1+q_i^{1-p}).
\end{equation*}
For $p\in(1,2)$, this implies
\begin{equation*}
    \big\|(I-E_{q_i})\delta_0\big\|_{\ell^p(\ZZ/q_i\Z)}\gtrsim_p1+q_i^{1-p}.
\end{equation*}
Taking natural logarithms and arguing as in \cite{KMSZ}, we obtain
\begin{equation}\label{mertens estimate}
    \prod_{i=1}^K(1+q_i^{1-p})\gtrsim \frac{q_K^{2-p}}{\log q_K}.
\end{equation}

Our starting observation is that $f=\delta_0$ does not maximize the norm, even when $p=2$. Indeed,
\begin{equation*}
    \big\|(I-E_{q_i})\delta_0\big\|_{\ell^2(\Z/q_i\Z)}=(1-q_i^{-1})^{1/2},
\end{equation*}
whereas $I - E_{q_i}$ has norm $1$ and, for example,  $f=\delta_0-\delta_1$ is a maximizing function.

The key difference in our approach is that we use two-valued functions to show that, for each $p\in(1,\infty)\setminus\{2\}$,
the operator norm of $I-E_{q_i}$ is uniformly larger than $1$ for all sufficiently large $q_i$. More precisely, we show that there exists $\eta_p>0$, independent of $q_i$, such that
\begin{equation}\label{lower bound eq1}
    \|I-E_{q_i}\|_{\ell^p(\Z/q_i\Z)\to\ell^p(\Z/q_i\Z)}
    >1+\eta_p
\end{equation}
for all sufficiently large $q_i$. Taking the product of these
lower bounds yields
\begin{equation*}
    \prod_{i=1}^K(1+\eta_p)=(1+\eta_p)^K,
\end{equation*}
which leads to an improvement over \eqref{mertens estimate}.

\subsection{Organization and structure}
The paper is organized as follows. In Section~\ref{sec:transference},
we show that the norm in \eqref{eq:main lower} can be bounded
from below by the norm of an operator being the Fourier projection onto reduced fractions with one fixed denominator. Similar idea was used in \cite{KMSZ} but here we use the language of the finite groups. In Section~\ref{sec:structure},
we factorize this Fourier projection by using the Chinese remainder theorem to get that this operator is related to the kernel \eqref{eq:kernel operator}.
Finally, in Section~\ref{sec:lower bound}, we prove the estimate
\eqref{lower bound eq1}.

\subsection{Acknowledgments}
The author is grateful to Dariusz Kosz and Tomasz Z. Szarek for reading an early version of the manuscript and for their precise comments and remarks, which have significantly improved this paper.

\subsection{Notation}
We write
\begin{equation*}
     \N:=\{1,2,\ldots\},\qquad\Z,\ \R,\ \C
\end{equation*}
for the positive integers, integers, real numbers, and complex
numbers, respectively. We denote the torus by $\T:=\R/\Z$ and identify it with $[0,1)$, equipped with addition modulo $1$.

For each $q\in\N$, we identify the quotient group $\Z/q\Z$
with the set
\begin{equation*}
    \{1,2,\ldots,q\},
\end{equation*}
equipped with addition modulo $q$.

Similarly, we identify its dual group with $q^{-1}\Z/\Z$,
for which we use the representatives
\begin{equation*}
    \Big\{\frac{1}{q},\frac{2}{q},\ldots,1\Big\},
\end{equation*}
with addition modulo $1$. In these representations, $q$ and $1$
represent the zero elements of $\Z/q\Z$ and $q^{-1}\Z/\Z$,
respectively.

If $A$ is a set or $S$ is a statement, then $\ind{A}$ and
$\ind{S}$ denote the corresponding indicator functions. The
cardinality of a finite set $A$ is denoted by $|A|$. For integers
$a$ and $q$, we write $(a,q)$ for their greatest common divisor.

For two quantities $A,B>0$, we write $A \lesssim B$ or $B \gtrsim A$ if $A\le CB$ for some $C>0$. If $A \lesssim B\lesssim A$, then we write $A \simeq B$. We will use the symbols $\lesssim_{\delta}$ or $\simeq_{\delta}$ to emphasize that the implicit constant $C$ depends on a parameter $\delta$.

\subsubsection{Groups and function spaces}

Let $(\G,+)$ be a locally compact abelian (LCA) group equipped
with a Haar measure $\lambda_{\G}$. For each $p\in[1,\infty)$,
let $L^p(\G)$ denote the space of measurable functions on $\G$,
identified up to equality almost everywhere, for which
\begin{equation*}
    \|f\|_{L^p(\G)}:=\Big(\int_{\G}|f(x)|^p\,\mathrm{d}\lambda_{\G}(x)\Big)^{1/p}
\end{equation*}
is finite.

In this paper, we are primarily interested in two cases.
For $\G=\Z$, we use counting measure and denote the corresponding
space by $\ell^p(\Z)$. For $\G=\Z/q\Z$, we use normalized counting
measure $\lambda_{\Z/q\Z}$ and denote the corresponding space
by $\ell^p(\Z/q\Z)$, with norm
\begin{equation*}
    \|f\|_{\ell^p(\Z/q\Z)}:=\Big(\frac{1}{q}\sum_{x\in\Z/q\Z}|f(x)|^p\Big)^{1/p}.
\end{equation*}

\subsubsection{Fourier transform}
We shall write $e(z)=e^{2\pi {\bf i} z}$ for every $z\in\C$, where ${\bf i}^2=-1$.  As before let $(\G, +)$ be an LCA group equipped with a Haar measure $\lambda_{\G}$. It is well-known that every LCA group $\G$ has a Pontryagin dual $\hat \G = (\hat \G,+)$, an LCA group with a Haar measure $\lambda_{\hat \G}$ and a pairing $\G \times \hat \G\ni (x,\xi) \mapsto \langle x , \xi\rangle\in\TT$, such that the Fourier transform $\calF_{\G} \colon L^1(\G) \to L^\infty(\hat \G)$ given by
\begin{equation*}
     \calF_{\G} f(\xi) := \int_{\G} f(x) e(-\langle x , \xi\rangle)\,{\rm d}\lambda_{\G}(x), \qquad \xi\in\hat \G,
\end{equation*}
extends to a unitary map from $L^2(\G)$ to $L^2(\hat \G)$; in particular, we have Plancherel's identity
\begin{equation*}
    \|\calF_\G f\|_{L^2(\hat \G)}=\|f\|_{L^2(\G)}, \qquad f\in L^2(\G).
\end{equation*}
We will use the inverse Fourier transform $\calF_\G^{-1} \colon L^2(\hat \G) \to L^2(\G)$ which is given by
\begin{equation*}
     \calF_\G^{-1} f(x) = \int_{\hat \G} f(\xi) e(\langle x , \xi\rangle)\,{\rm d}\lambda_{\hat \G}(\xi), \qquad x\in\G,
\end{equation*}
for any  $f\in L^1(\hat \G) \cap L^2(\hat \G)$.

We list the pairs $(\G,\hat \G)$ of Pontryagin dual LCA groups we shall work with:
\begin{itemize}
\item [(i)]  If $\G = \R$ with Lebesgue measure $\lambda_{\R} = {\rm d}x$, then $\hat \G = \R$ with Lebesgue measure $\lambda_{\R} = {\rm d}\xi$ is a Pontryagin dual, with pairing $\langle x , \xi\rangle := x \cdot \xi$. Further, for $f \in L^1(\R)$,
\begin{align*}
\mathcal F_{\R} f(\xi)  :=  \int_{\R} f(x) e(-x\cdot\xi) \,{\rm d}x, \qquad \xi\in\R.
\end{align*}
\item[(ii)]  If $\G = \Z$ with counting measure $\lambda_{\Z}$, then $\hat \G = \TT$ with Lebesgue measure $\lambda_{\TT} = {\rm d}\xi$ is a Pontryagin dual, with pairing $\langle x , \xi\rangle :=x \cdot \xi$ and for  $f \in \ell^1(\Z)$,
\begin{align*}
\mathcal F_{\Z}f(\xi) :=\sum_{x \in \Z} f(x) e(-x\cdot\xi), \qquad \xi\in \TT.
\end{align*}
\item[(iii)] If $\G = \Z/q\Z$ for some $q \in \N$ with normalized counting measure $\lambda_{\Z/q\Z}$, then $\hat{\G} = q^{-1}\Z/\Z$ with counting measure $\lambda_{q^{-1}\Z/\Z}$ is a Pontryagin dual, with pairing $\langle x , \xi\rangle :=x \cdot \xi$. Also, for any $f \in \ell^1(\Z/q\Z)$,
\begin{align*}
\mathcal F_{\Z/q\Z}f(\xi):= q^{-1}\sum_{x \in \Z/q\Z} f(x) e(-x\cdot\xi), \qquad \xi\in q^{-1}\Z/\Z.
\end{align*}
\end{itemize}
\section{Transference from \texorpdfstring{$\Z/Q\Z$}{Z/QZ} to \texorpdfstring{$\Z$}{Z}}\label{sec:transference}
Fix a large integer $Q\in\N$ and set $N:=Q-1$. Consider $\Z/Q\Z$ with normalized counting measure and define the multiplier
\begin{equation*}
    m_N^{(Q)}(a/Q):=\ind{\{Q/(a,Q)\leq N\}},\quad a/Q \in Q^{-1} \Z / \Z.
\end{equation*}
Let $P_N^{(Q)}$ denote the associated Fourier multiplier operator,
defined for $F\in\ell^1(\Z/Q\Z)$ by
\begin{equation*}
    P_N^{(Q)}F(x):=\sum_{a/Q\in Q^{-1}\Z/\Z}
    m_N^{(Q)}(a/Q)\mathcal{F}_{\Z/Q\Z}F(a/Q)\,e(x\cdot a/Q),\qquad x\in\Z/Q\Z.
\end{equation*}
In other words, $P_N^{(Q)}$ is the Fourier projection onto the
elements in $Q^{-1}\Z/\Z$ whose reduced denominators are
at most $N$.

Let $\mathfrak{m}\in\mathcal{S}(\R)$ be supported in $[-1/2,1/2]$
and satisfy $\mathfrak{m}(0)=1$. Let $(\varepsilon_N)_{N\in\N}$
be any sequence satisfying $0<\varepsilon_N\leq N^{-2}$.
We will show that
\begin{equation}\label{eq:3:trans}
    \bigl\|
        \mathcal{F}_{\Z}^{-1}
        \bigl(\Delta_{\varepsilon_N,\mathcal{R}_{\leq N}}[\mathfrak{m}]
            \calF_{\Z}(\cdot)\bigr)\bigr\|_{\ell^p(\Z)\to\ell^p(\Z)}
    \geq
    \|P_N^{(Q)}\|_{\ell^p(\Z/Q\Z)\to\ell^p(\Z/Q\Z)}.
\end{equation}

First, choose a nonzero function $\varphi\in\mathcal{S}(\R)$ such that
$\supp\calF_{\R}\varphi\subseteq[-1,1]$. For $L\geq1$, define
\begin{equation}\label{eq:2:def}
    f_L(x):=\varphi(x/L)F(x\bmod Q),
    \qquad x\in\Z.
\end{equation}
Since $\varphi\in\mathcal{S}(\R)$, we have $f_L\in\ell^p(\Z)$ for every $p,L\in[1,\infty)$. Using the Fourier expansion of $F$, we can write
\begin{equation*}
    f_L(x)
    =\sum_{a/Q\in Q^{-1}\Z/\Z}
    \mathcal{F}_{\Z/Q\Z}F(a/Q)\varphi(x/L)e(x\cdot a/Q),
    \qquad x\in\Z.
\end{equation*}
By the Poisson summation formula,
\begin{equation*}
    \begin{aligned}
        \calF_{\Z}f_L(\xi)
        &:=\sum_{x\in\Z}f_L(x)e(-x\cdot\xi)=L\sum_{a/Q\in Q^{-1}\Z/\Z}\calF_{\Z/Q\Z}F(a/Q)
        \sum_{k\in\Z}
        \calF_{\R}(\varphi)\bigl(L(\xi-a/Q-k)\bigr),\qquad \xi\in\T.
    \end{aligned}
\end{equation*}
Thus, the Fourier transform of $f_L$ is supported in the union
of the intervals of radius $1/L$ centered at the frequencies
$a/Q\in Q^{-1}\Z/\Z$.

For $x\in\ZZ$, define
\begin{equation}\label{eq:phiLN}
    \varphi_{L,N}(x):=\calF_{\RR}^{-1}\bigl(\mathfrak{m}((\varepsilon_N L)^{-1}\cdot)\calF_{\RR}\varphi\bigr)(x).
\end{equation}
We will show that, for sufficiently large $L$,
\begin{equation}\label{eq:1form}
    \mathcal{F}_{\Z}^{-1}
    \bigl(\Delta_{\varepsilon_N,\mathcal{R}_{\leq N}}[\mathfrak{m}]\calF_{\Z}f_L\bigr)(x)
    =\varphi_{L,N}(x/L)(P_N^{(Q)}F)(x\bmod Q),\qquad x\in\Z.
\end{equation}
More precisely, choose $L$ sufficiently large that
\begin{equation*}
    \frac1L<\min\Big\{\frac{\varepsilon_N}{2},\frac1{QN}-\frac{\varepsilon_N}{2}\Big\}.
\end{equation*}
By assumption on $\varepsilon_N$, the second number in the minimum is positive
for $N\geq3$, since $Q=N+1$ and
\begin{equation*}
    \frac{\varepsilon_N}{2}\leq\frac1{2N^2}<\frac1{N(N+1)}=\frac1{QN}.
\end{equation*}

First, suppose that $a/Q$ has denominator $Q=N+1$.
Then $a/Q\notin\mathcal{R}_{\leq N}$. Consequently, for every
$b/q\in\mathcal{R}_{\leq N}$, the points $a/Q$ and $b/q$
are distinct. Moreover,
\begin{equation*}
    \left|\frac aQ-\frac bq\right|\geq\frac1{Qq}\geq\frac1{QN}.
\end{equation*}
Every $\xi$ in the portion of the support of $\calF_{\Z}f_L$
around $a/Q$ lies within distance $1/L$ of $a/Q$.
By the choice of $L$, this radius is smaller than
\begin{equation*}
    \frac1{QN}-\frac{\varepsilon_N}{2}.
\end{equation*}
Therefore, such a point $\xi$ cannot belong to the support of $\Delta_{\varepsilon_N,\mathcal{R}_{\leq N}}[\mathfrak{m}]$.
Thus, $\Delta_{\varepsilon_N,\mathcal{R}_{\leq N}}[\mathfrak{m}](\xi)=0$ for $\xi$ in these intervals around $a/Q$ with $(a,Q)=1$.

Next, suppose that $a/Q$ has reduced denominator at most $N$.
Then $a/Q$ is a canonical fraction, that is,
$a/Q\in\mathcal{R}_{\leq N}$.
Using the assumption that $(\varepsilon_N L)^{-1}<1/2$, we obtain
\begin{equation*}
    \mathcal{F}_{\Z}^{-1}
    \bigl(\Delta_{\varepsilon_N,\mathcal{R}_{\leq N}}[\mathfrak{m}]
        \calF_{\Z}f_L\bigr)(x)
    =\varphi_{L,N}(x/L)(P_N^{(Q)}F)(x\bmod Q).
\end{equation*}
This identity follows from the change of variables
$\eta:=L(\xi-a/Q)$ in the integral defining $\calF_{\Z}^{-1}$.
This proves \eqref{eq:1form}.

Now, we are ready to prove \eqref{eq:3:trans}. This estimate also can be deduced from the classical results of de Leeuw; see~\cite[Lemma 4.1]{deLeeuw}. The arguments here are standard. Let $H\colon \Z/Q\Z\to\C$ and let $h_L$ be defined analogusly as in \eqref{eq:2:def}. By splitting the sum into residue classes modulo $Q$, we obtain
\begin{equation*}
    \frac{1}{L}\|h_L\|^p_{\ell^p(\ZZ)}=\sum_{y\in\Z/Q\Z}|H(y)|^p\frac{1}{L}\sum_{x\in\Z}|\varphi(L^{-1}(xQ+y))|^p
\end{equation*}
For every fixed $y\in \Z/Q\Z$, the inner sum is a Riemann sum, so one has
\begin{equation*}
    \lim_{L\to\infty}\frac{1}{L}\sum_{x\in\Z}|\varphi(L^{-1}(xQ+y))|^p=\frac{1}{Q}\|\varphi\|^p_{L^p(\R)}
\end{equation*}
and consequently
\begin{equation*}
    \lim_{L\to\infty}\frac{1}{L}\|h_L\|^p_{\ell^p(\ZZ)}=\|\varphi\|^p_{L^p(\R)}\|H\|_{\ell^p(\Z/Q\Z)}^p
\end{equation*}
The same argument applies to $\varphi_{L,N}$ defined in
\eqref{eq:phiLN}. Indeed, fix $N$ and set
$g_L(t):=|\varphi_{L,N}(t)|^p$. By dividing the integral into disjoint intervals and using the fundamental theorem of calculus we get
\begin{align*}
    &\Big|
        \frac{Q}{L}\sum_{x\in\Z}g_L\big(L^{-1}(Qx+y)\big)
        -\int_{\R}g_L(t)\,\mathrm{d}t\Big|\leq\frac{Q}{L}\|g_L'\|_{L^1(\R)}.
\end{align*}
Moreover, since $\varphi_{L,N}\to\mathfrak{m}(0)\varphi$, as $L\to\infty$, in $\mathcal{S}(\R)$ we get that $\|g_L'\|_{L^1(\R)}$ is uniformly bounded and consequently
\begin{equation*}
    \lim_{L\to\infty}
    \frac{1}{L}\sum_{x\in\Z}
    \Big|\varphi_{L,N}\big(L^{-1}(Qx+y)\big)\Big|^p=\frac{1}{Q}\|\varphi\|_{L^p(\R)}^p.
\end{equation*}

Applying the above first to $H=F$ and then to
$H=P_N^{(Q)}F$, by~\eqref{eq:1form}, we get
\begin{equation*}
    \lim_{L\to\infty}\frac{\big\|\mathcal{F}_{\Z}^{-1}
        \big(\Delta_{\varepsilon_N,\mathcal{R}_{\leq N}}[\mathfrak{m}]
            \calF_{\Z}f_L\big)\big\|_{\ell^p(\Z)}}{\|f_L\|_{\ell^p(\Z)}}=
\frac{\|P_N^{(Q)}F\|_{\ell^p(\Z/Q\Z)}}{\|F\|_{\ell^p(\Z/Q\Z)}},
\end{equation*}
for every nonzero $F\in \ell^p(\Z/Q\Z)$. This implies
\begin{equation*}
    \big\|\mathcal{F}_{\Z}^{-1}
        \big(
            \Delta_{\varepsilon_N,\mathcal{R}_{\leq N}}[\mathfrak{m}]
            \calF_{\Z}(\cdot)\big)\big\|_{\ell^p(\Z)\to\ell^p(\Z)}\geq \frac{\|P_N^{(Q)}F\|_{\ell^p(\Z/Q\Z)}}
     {\|F\|_{\ell^p(\Z/Q\Z)}}.
\end{equation*}
Taking the supremum over all $F$ yields~\eqref{eq:3:trans}. Therefore in order to show Theorem~\ref{thm:IW:lower_new} it suffices to show
\begin{equation*}
    \|P_N^{(Q)}\|_{\ell^p(\Z/Q\Z)\to \ell^p(\Z/Q\Z)}\gtrsim_p N^{c_p/\log\log N},
\end{equation*}
for some constant $c_p>0.$

\section{Structure of \texorpdfstring{$P_N^{(Q)}$}{PN(Q)} when \texorpdfstring{$Q$}{Q} is a product of distinct primes}\label{sec:structure}

We now analyze the operator $P_N^{(Q)}$ when $Q$
is a product of distinct primes. We first reduce $P_{Q-1}^{(Q)}$ to
the projection onto fractions with denominator $Q$, and then use the
Chinese remainder theorem to reduce the latter product of one-dimensional operators.

Let
\begin{equation*}
   Q= q_1\cdots q_K, 
\end{equation*}
where $q_1,\ldots,q_K$ are distinct primes, and set
$N:=Q-1$. Let $D_Q$ be the Fourier multiplier operator defined by
\begin{equation*}
D_QF(x):=\sum_{\substack{a/Q\in Q^{-1}\Z/\Z\\(a,Q)=1}}
\mathcal{F}_{\Z/Q\Z}F(a/Q)e(x\cdot a/Q),\quad x\in\Z/Q\Z.
\end{equation*}
We are going to prove that the convolution kernel of the above operator is given by 
\begin{equation*}
    \prod_{i=1}^K(q_i\mathds{1}_{q_i|x}-1),\quad x\in\Z/Q\Z.
\end{equation*}
The above operator acts on $F_a(x):=e(x\cdot a/Q)$ in the following way
\begin{equation*}
    D_QF_a=\begin{cases}
        F_a,&\text{ if }(a,Q)=1,\\
        0,&\text{ if }(a,Q)>1.
    \end{cases}
\end{equation*}

Since every proper divisor $d<Q$ satisfies
$d\leq Q/2<Q-1$, we have $P_{Q-1}^{(Q)}=I-D_Q.$
Consequently,
\begin{equation}\label{eq:P-from-D}
\|P_{Q-1}^{(Q)}\|_{\ell^p(\Z/Q\Z)\to \ell^p(\Z/Q\Z)}
\geq
\|D_Q\|_{\ell^p(\Z/Q\Z)\to \ell^p(\Z/Q\Z)}-1.
\end{equation}

Now we are going to show the factorization of the operator $D_Q$. Set $Q_j:=\frac{Q}{q_j},$
and choose $b_j\in\Z$ such that
\begin{equation*}
    b_jQ_j\equiv1\pmod{q_j}.
\end{equation*}
Existence of such $b_j$ follows from B\'ezout's lemma since $q_j$ and $Q_j$ are co-prime. By the Chinese remainder theorem, every $x\in\Z/Q\Z$ is uniquely determined by the tuple
\begin{equation*}
    (x_1,\ldots,x_K)\in\prod_{j=1}^K\Z/q_j\Z,\qquad x_j\equiv x\pmod{q_j},
\end{equation*}
and one has
\begin{equation*}
    x\equiv\sum_{j=1}^K x_jb_jQ_j\pmod Q.
\end{equation*}

It follows that
\begin{equation*}
e\Big(\frac{ax}{Q}\Big)=e\Big(\frac aQ\sum_{j=1}^Kx_jb_jQ_j\Big)=\prod_{j=1}^K
e\Big(\frac{ab_jx_j}{q_j}\Big).
\end{equation*}
Now, if $a_j\equiv ab_j\pmod{q_j},$ we have
\begin{equation*}
    e\Big(\frac{ax}{Q}\Big)=\prod_{j=1}^Ke\Big(\frac{a_jx_j}{q_j}\Big).
\end{equation*}
Since $b_j$ is invertible modulo $q_j$, we have $a_j\neq0\pmod{q_j}$ if and only if $q_j\nmid a$. Since $Q=q_1\cdots q_K$, this gives
\begin{equation*}
    (a,Q)=1\quad\Longleftrightarrow\quad a_j\neq0\pmod{q_j}\quad\text{for every }j.
\end{equation*}

For each $j\in\{1,\ldots,K\}$, let
\begin{equation*}
    E_{q_j}f:=\frac1{q_j}\sum_{y\in\Z/q_j\Z}f(y).
\end{equation*}
By ortogonality, we have
\begin{equation*}
    (I-E_{q_j})e\Big(\frac{a_jx_j}{q_j}\Big)=
\begin{cases}
e\Big(\dfrac{a_jx_j}{q_j}\Big),
    & a_j\neq0,\\[4pt]
0,  & a_j=0.
\end{cases}
\end{equation*}

Therefore, the tensor operator given by
$\bigotimes_{j=1}^K(I-E_{q_j})$
preserves the exponential $e(ax/Q)$ precisely when
$(a,Q)=1$, and annihilates it otherwise. This is exactly the action
of $D_Q$. Since the exponential functions form a basis of the space of functions on $\Z/Q\Z$ we can write
\begin{equation}\label{eq:D-tensor-product}
D_Q=\bigotimes_{j=1}^K(I-E_{q_j}).
\end{equation}
Finally, by factorizing $\Z/Q\Z$ into product of $\Z/q_i\Z$ and taking the product of maximizers, we see that for $p\in(1,\infty)$, one has
\begin{equation}\label{eq:lower:est}
    \|D_Q\|_{\ell^p(\Z/Q\Z)\to \ell^p(\Z/Q\Z)}
\geq
\prod_{j=1}^K
\|I-E_{q_j}\|_
{\ell^p(\Z/q_j\Z)\to \ell^p(\Z/q_j\Z)}.
\end{equation}

In the next section, we prove that for every
$p\in(1,\infty)\setminus\{2\}$, there exist constants
$\eta_p>0$ and $q_p\geq2$ such that
\begin{equation}\label{eq:one-dimensional-lower}
\|I-E_q\|_{\ell^p(\Z/q\Z)\to \ell^p(\Z/q\Z)}
\geq 1+\eta_p
\end{equation}
for every prime $q\geq q_p$.

Assuming~\eqref{eq:one-dimensional-lower}, choose distinct primes
$q_1,\ldots,q_K\geq q_p$ and write
\begin{equation*}
    Q_K=q_1\cdots q_K,
\qquad
N_K=Q_K-1.
\end{equation*}
Estimate \eqref{eq:lower:est} gives
\begin{equation*}
    \|D_{Q_K}\|_{\ell^p(\Z/Q_K\Z)\to \ell^p(\Z/Q_K\Z)}\geq(1+\eta_p)^K.
\end{equation*}
Combining the above with \eqref{eq:P-from-D}, we obtain
\begin{equation}\label{eq:P-lower-K}
\|P_{N_K}^{(Q_K)}\|_{\ell^p(\Z/Q_K\Z)\to \ell^p(\Z/Q_K\Z)}\geq(1+\eta_p)^K-1.
\end{equation}

Finally, let $(p_n)_{n\in\N}$ denote the increasing sequence of the
prime numbers. Choose $j_p\in\N$ so that $p_{j_p}\geq q_p$, and
set
\begin{equation*}
    q_j:=p_{j_p+j-1},\qquad 1\leq j\leq K.
\end{equation*}
Thus, $q_1,\ldots,q_K$ are $K$ consecutive primes, all of
which are at least $q_p$. Recall that $Q_k=q_1\dots q_K$. Taking logarithms gives $\log Q_K=\sum_{j=1}^K\log q_j=\vartheta(q_K)-\vartheta(p_{j_p-1}),$ where $\vartheta$ is Chebyshev's theta function. Since $j_p$ is fixed, the second term
is independent of $K$. By the prime number theorem,
\begin{equation*}
    \lim_{x\to\infty}\frac{\vartheta(x)}{x}=1
\end{equation*}
and hence $\log Q_K\eqsim q_K.$
Again, by the prime number theorem, we get $q_K\simeq K\log K$ because $j_p$ is fixed. Consequently,
\begin{equation}\label{eq:log-QK-asymptotic}
\log Q_K\simeq_p K\log K.
\end{equation}
Since $N_K=Q_K-1$, we get
\begin{equation}\label{eq:K-vs-NK}
K\simeq_p\frac{\log N_K}{\log\log N_K}.
\end{equation}

On the other hand,~\eqref{eq:P-lower-K} gives
\begin{equation*}
    \|P_{N_K}^{(Q_K)}\|_{\ell^p(\Z/Q_K\Z)\to \ell^p(\Z/Q_K\Z)}\geq(1+\eta_p)^K-1.
\end{equation*}
Combining this with~\eqref{eq:K-vs-NK}, we get that there exists a constant $c_p>0$ such that
\begin{equation*}
    \|P_{N_K}^{(Q_K)}\|_{\ell^p(\Z/Q_K\Z)\to \ell^p(\Z/Q_K\Z)}\geq\exp\Big(
c_p\frac{\log N_K}{\log\log N_K}\Big)
\end{equation*}
for every sufficiently large $K$. This proves Theorem~\ref{thm:IW:lower_new}.
\section{Estimate for the one-dimensional projection -- proof of \eqref{eq:one-dimensional-lower}}\label{sec:lower bound}

It remains to prove \eqref{eq:one-dimensional-lower}. For any $q\in\N$ and $f\in \ell^1(\Z/q\Z)$ let
\begin{equation*}
    E_qf:=\frac1q\sum_{x\in\Z/q\Z}f(x)
\end{equation*}
be the projection onto the constant functions on $\Z/q\Z$,
equipped with normalized counting measure.

On $\ell^2(\Z/q\Z)$, the operator $E_q$ is an orthogonal
projection, and therefore so is $I-E_q$. Hence
\begin{equation*}
    \|f\|_{\ell^2(\Z/q\Z)}^2=\|E_q f\|_{\ell^2(\Z/q\Z)}^2+\|(I-E_q)f\|_{\ell^2(\Z/q\Z)}^2,
\end{equation*}
which gives
\begin{equation*}
    \|I-E_q\|_{\ell^2(\Z/q\Z)\to\ell^2(\Z/q\Z)}\leq1.
\end{equation*}
Taking, for example, $f=\mathds{1}_{\{0\}}-\mathds{1}_{\{1\}},$ we have $E_q f=0$, and thus
\begin{equation*}
    \|I-E_q\|_{\ell^2(\Z/q\Z)\to \ell^2(\Z/q\Z)}=1.
\end{equation*}

We now show that, when $p\neq2$, the norm of $I-E_q$ is
uniformly larger than $1$ for all sufficiently large~$q$.

\begin{lemma}\label{lem:one-dimensional-lower}
Let $p\in(1,\infty)\setminus\{2\}$. There exist constants
$\eta_p>0$ and $q_p\geq2$ such that
\begin{equation*}
    \|I-E_q\|_{\ell^p(\Z/q\Z)\to \ell^p(\Z/q\Z)}\geq1+\eta_p
\end{equation*}
for every integer $q\geq q_p$. More precisely, for any fixed
$\theta\in(0,1)\setminus\{1/2\}$, one may take
\begin{equation*}
    \eta_p:=\frac{C_p(\theta)-1}{2},
\end{equation*}
where
\begin{equation*}
    C_p(\theta):=\big(\theta^{p-1}+(1-\theta)^{p-1}\big)^{1/p}\big(\theta^{1/(p-1)}+(1-\theta)^{1/(p-1)}\big)^{(p-1)/p}.
\end{equation*}
\end{lemma}

\begin{proof}
We are going to find maximizers along functions of the form
\begin{equation*}
    g(n):=x\mathds{1}_A(n)+y\mathds{1}_{A^c}(n),\quad x,y\in\C,\quad n\in\Z/q\Z,
\end{equation*}
for $A\subseteq \Z/q\Z$. Note that
\begin{equation*}
    E_q g=\frac{|A|}{q}x+\Big(1-\frac{|A|}{q}\Big)y=\delta_A x+(1-\delta_A)y
\end{equation*}
with $\delta_A:=|A|/q\in[0,1]$ being the density of the set $A$.

We first compute the norm of the corresponding projection on a
two-point probability space. Let
\begin{equation*}
    X_\theta=\{0,1\},
\qquad
\mu_\theta(\{0\})=\theta,
\qquad
\mu_\theta(\{1\})=1-\theta.
\end{equation*}
We identify functions $f \in L^p(X_\theta)$ with pairs $(f(0),f(1))$. For $f=(x,y)\in L^p(X_\theta)$, define
\begin{equation*}
    \mathbb E_\theta f:=\theta x+(1-\theta)y.
\end{equation*}
Then
\begin{equation*}
    (I-\mathbb E_\theta)f=\bigl((1-\theta)(x-y),-\theta(x-y)\bigr),
\end{equation*}
and consequently
\begin{equation*}
    \|(I-\mathbb E_\theta)f\|_{L^p(X_\theta)}=|x-y|[\theta(1-\theta)^p+(1-\theta)\theta^p]^{1/p}.
\end{equation*}
Let $L(x,y)=x-y$. With respect to the duality pairing
\begin{equation*}
    \langle (x,y),(u,v)\rangle:=\theta xu+(1-\theta)yv,
\end{equation*}
the functional $L\colon L^p(X_\theta)\to\C$ is represented by
\begin{equation*}
    \Big(\frac1\theta,-\frac1{1-\theta}\Big).
\end{equation*}
If $p'=p/(p-1)$, then the duality of $L^p$ gives,
\begin{equation*}
    \|L\|=\left[\theta^{-1/(p-1)}+(1-\theta)^{-1/(p-1)}\right]^{(p-1)/p}.
\end{equation*}
Combining the two formulas gives
\begin{equation}\label{eq:two-point-exact-norm}
\|I-\mathbb E_\theta\|_{L^p(X_\theta)\to L^p(X_\theta)}
=
C_p(\theta).
\end{equation}

Now we will show that $C_p(\theta)>1$ whenever $p\neq2$ and
$\theta\neq1/2$. H\"older's inequality gives
\begin{equation*}
\begin{aligned}
1&=\theta+(1-\theta)
\\
&\leq\big(\theta^{p-1}+(1-\theta)^{p-1}\big)^{1/p}\big(\theta^{1/(p-1)}+
(1-\theta)^{1/(p-1)}\big)^{(p-1)/p}
\\
&=C_p(\theta).
\end{aligned}
\end{equation*}
Equality would imply
\begin{equation*}
    \theta^{p-1-1/(p-1)}=(1-\theta)^{p-1-1/(p-1)}.
\end{equation*}
Since $p\neq2$, the exponent is nonzero, and therefore equality forces $\theta=1/2$. This proves that $C(\theta)>1.$

We now transfer this two-point calculation to $\Z/q\Z$. Set
\begin{equation*}
    m_q:=\lfloor\theta q\rfloor,\qquad\theta_q:=\frac{m_q}{q}.
\end{equation*}
For all sufficiently large $q$, we have
$0<m_q<q$, and $\theta_q\to\theta.$ Next, we choose a subset $A_q\subseteq\Z/q\Z$ with
$|A_q|=m_q$, and let
\begin{equation*}
    V_q:=\big\{x\mathds{1}_{A_q}+y\mathds{1}_{A_q^c}\colon x,y\in\C\big\}
\end{equation*}
be the subspace of $\ell^p(\Z/q\Z)$. For $x,y\in\C$ we define the function on $\Z/q\ZZ$ by
\begin{equation*}
    U_q(x,y)(n):=x\mathds{1}_{A_q}(n)+y\mathds{1}_{A_q^c}(n),\quad n\in\Z/q\Z.
\end{equation*}

We are going to show that $U_q$ defines an isometric isomorphism between $L^p(X_{\theta_q})$ and $V_q$. It is easy to see that the function is linear and invertible. Further, we have
\begin{equation*}
    \|U_q(x,y)\|_{\ell^p(\Z/q\Z)}^p=\theta_q|x|^p+(1-\theta_q)|y|^p=\|(x,y)\|_{L^p(X_{\theta})}^p,
\end{equation*}
so $U_q$ is an isometric isomorphism from $L^p(X_{\theta_q})$ onto $V_q$. Moreover,
\begin{equation*}
    E_q U_q(x,y)=\theta_q x+(1-\theta_q)y,
\end{equation*}
viewed as a constant function on $\Z/q\Z$. Hence
\begin{equation*}
    (I-E_q)U_q=U_q(I-\mathbb E_{\theta_q}),
\end{equation*}
and the restriction of $I-E_q$ to $V_q$ is isometrically equivalent to $I-\mathbb E_{\theta_q}$. By~\eqref{eq:two-point-exact-norm},
\begin{equation*}
    \|I-E_q\|_{\ell^p(\Z/q\Z)\to \ell^p(\Z/q\Z)}\geq C_p(\theta_q).
\end{equation*}
Since $C_p$ is continuous on $(0,1)$ and $\theta_q\to\theta$, we have
\begin{equation*}
    C_p(\theta_q)\to C_p(\theta)>1.
\end{equation*}
Set $\eta_p:=\frac{C_p(\theta)-1}{2}.$ Then, for all sufficiently large $q$, one has $ C_p(\theta_q)\geq1+\eta_p,$ which proves the claim.
\end{proof}
\subsection{Final Remarks}
The constant $C_p(\theta)$ appearing in Lemma~\ref{lem:one-dimensional-lower}
has a remarkably elegant form. It is known as the Franchetti constant
and is well-known in approximation and optimization theory in Banach
spaces; see the article by Shargorodsky and Sharia~\cite{SS}
for an introduction to the problem and its history.

In fact, a substantially stronger result than
Lemma~\ref{lem:one-dimensional-lower} is available.
Namely, Lewicki and Skrzypek~\cite[Lemma~2.9]{LS} proved that
\begin{equation*}
    \lim_{q\to\infty}
    \|I-E_q\|_{\ell^p(\Z/q\Z)\to\ell^p(\Z/q\Z)}
    =
    \sup_{\theta\in(0,1)}C_p(\theta).
\end{equation*}
Although our result is a way weaker, it suffices
for our purposes.

\end{document}